\documentclass[sn-mathphys,Numbered]{sn-jnl}

\usepackage{graphicx}%
\usepackage{multirow}%
\usepackage{amsmath,amssymb,amsfonts}%
\usepackage{amsthm}%
\usepackage{mathrsfs}%
\usepackage[title]{appendix}%
\usepackage{xcolor}%
\usepackage{textcomp}%
\usepackage{manyfoot}%
\usepackage{booktabs}%
\usepackage{algorithm}%
\usepackage{algorithmicx}%
\usepackage{algpseudocode}%
\usepackage{listings}%

\theoremstyle{thmstyleone}%
\newtheorem{theorem}{Theorem}
\newtheorem{proposition}[theorem]{Proposition}%

\theoremstyle{thmstyletwo}%
\newtheorem{remark}{Remark}%
\newtheorem{corollary}{Corollary}%
\newtheorem{lemma}{Lemma}%

\theoremstyle{thmstylethree}%

\begin{document}

\title[Article Title]{Weak convergence of probability measures on hyperspaces with the upper Fell-topology}


\author{\fnm{Dietmar} \sur{Ferger}}\email{dietmar.ferger@tu-dresden.de}



\affil{\orgdiv{Fakult\"{a}t Mathematik}, \orgname{Technische Universit\"{a}t Dresden}, \orgaddress{\street{Zellescher Weg 12-14}, \city{Dresden}, \postcode{01069}, \country{Germany}}}




\abstract{Let $E$ be a locally compact second countable Hausdorff space and $\mathcal{F}$ the pertaining family of all closed sets. We endow $\mathcal{F}$ respectively with the Fell-topology, the upper Fell topology or the upper Vietoris-topology and investigate weak convergence of probability measures on the corresponding hyperspaces with a focus on the upper Fell topology. The results can be transferred to distributional convergence of random closed sets in $E$ with applications to the asymptotic behavior of measurable selection.}

\keywords{Weak convergence, hyperspaces, upper Fell topology, upper Vietoris topology, random closed sets, capacity functionals.}


\pacs[MSC Classification]{60B05,60B10,26E25.}

\maketitle

\section{Introduction}\label{sec1}
To begin with we introduce concepts, which are fundamental for the whole paper: the \emph{Fell-topology}, the \emph{upper Fell-topology} and \emph{weak convergence on topological spaces}.
To introduce the Fell-topologies let $E$ be a non-empty set endowed with a topology $\mathcal{G}$ and the pertaining families $\mathcal{F}$ and $\mathcal{K}$ of all
closed sets and all compact sets, respectively. For an arbitrary subset $A \subseteq E$ one defines
$$
 \mathcal{M}(A):=\{F \in \mathcal{F}: F \cap A = \emptyset\}
$$
and
$$
 \mathcal{H}(A):=\{F \in \mathcal{F}: F \cap A \neq \emptyset\}.
$$
The elements of $\mathcal{M}(A)$ or $\mathcal{H}(A)$ are called \emph{missing sets} or \emph{hitting sets}, respectively, of $A$.
Put
$$
\mathcal{S}_{F}= \{\mathcal{M}(K): K \in \mathcal{K}\} \cup \{\mathcal{H}(G): G \in \mathcal{G}\} \subseteq 2^\mathcal{F}.
$$
and
$$
\mathcal{S}_{uF}= \{\mathcal{M}(K): K \in \mathcal{K}\} \subseteq 2^\mathcal{F}.
$$
Then the topologies on $\mathcal{F}$ generated by $\mathcal{S}_F$ or $\mathcal{S}_{uF}$ are called \emph{Fell-topology} or \emph{upper Fell-topology} and are denoted by $\tau_F$ or $\tau_{uF}$, respectively. The name goes back to J. Fell (1962) \cite{Fell}. The topological spaces $(\mathcal{F}, \tau_{F})$ and $(\mathcal{F}, \tau_{uF})$ are examples of  \emph{hyperspaces}. From now on we assume that the underlying \emph{carrier space} $(E,\mathcal{G})$
is locally compact, second-countable and Hausdorff. Then the hyperspace $(\mathcal{F}, \tau_{F})$ is compact, second-countable and Hausdorff,
confer G. Beer \cite{Beer} for these properties and much more information on the Fell-topology.
Whereas $(\mathcal{F}, \tau_{uF})$ is also compact and second-countable, but it is not Hausdorff, confer Ferger \cite{Ferger0}.

To introduce the second concept let $(X,\mathcal{O})$ be an arbitrary topological space with induced Borel $\sigma$-algebra $\mathcal{B}\equiv\mathcal{B}(X) := \sigma(\mathcal{O})$. Let
$$
 \Pi\equiv\Pi(X,\mathcal{O}):=\{P: P \text{ is a probability measure on } (X,\mathcal{B})\}.
$$
For each open $O \in \mathcal{O}$ consider the \emph{evaluation map} $e_O:\Pi \rightarrow \mathbb{R}$ defined by $e_O(P):=P(O)$ for all $P \in \Pi$.
If $\mathcal{O}_>:=\{(a,\infty): a \in \mathbb{R}\} \cup \{\emptyset,\mathbb{R}\}$ is the \emph{right-order topology} on $\mathbb{R}$, then the initial
topology with respect to the functions $e_O:\Pi \rightarrow (\mathbb{R},\mathcal{O}_>), \; O \in \mathcal{O},$ is called \emph{weak topology} on
$\Pi$ and denoted by $\tau_{weak}$. It goes back to Tops{\o}e \cite{Top}. An introduction to the weak topology can be found in the textbook
of G\"{a}nssler and Stute \cite{Stute}.

Next, let $(A,\le)$ be a directed set and $(P_\alpha)$ be a net in $\Pi$ converging to $P \in \Pi$ in the weak topology:
\begin{equation} \label{wc}
 P_\alpha \rightarrow P \; \text{ in } (\Pi,\tau_{weak}).
\end{equation}
By construction of the initial topology we have that (\ref{wc}) is equivalent to
$$
 e_O(P_\alpha) \rightarrow e_O(P) \; \text{ in } (\mathbb{R},\mathcal{O}_>) \quad \forall \; O \in \mathcal{O},
$$
which in turn by the definitions of $e_O$ and $\mathcal{O}_>$ is equivalent to
\begin{equation} \label{open}
 \liminf_\alpha P_\alpha(O) \ge P(O) \quad \forall \; O \in \mathcal{O}.
\end{equation}
By complementation we further obtain from (\ref{open}) that (\ref{wc}) is equivalent to
\begin{equation}\label{closed}
 \limsup_\alpha P_\alpha(F) \le P(F) \quad \text{for all closed sets } F \text{ in } (X,\mathcal{O}).
\end{equation}
Thus, if $(X,\mathcal{O})$ is metrizable and $P_\alpha$ actually is a sequence, then one can conclude from the Portmanteau-Theorem
that convergence in $(\Pi, \tau_{weak})$ is the same as the well-known \emph{weak convergence} of probability measures on metric spaces, confer, e.g. Billingsley \cite{Bill2}. Therefore, in general we say that
$(P_\alpha)$ \emph{converges weakly} to $P$ on $(X,\mathcal{O})$, if (\ref{wc}) holds and alternatively write for this: $P_\alpha \rightarrow_w P$ on $(X,\mathcal{O})$.

We are now in a position to explain what this paper is about. It is (mainly) about the characterisation of weak convergence $P_\alpha \rightarrow_w P$ on $(\mathcal{F},\tau_{uF})$ and its relation to weak convergence on $(\mathcal{F},\tau_{F})$. Let
$\mathcal{B}_F:=\sigma(\tau_F)$ and $\mathcal{B}_{uF}=\sigma(\tau_{uF})$ be the underlying Borel-$\sigma$ algebras.
It follows from Lemma 2.1.1 in Schneider and Weil \cite{Weil} that these coincide:
\begin{equation} \label{Borel}
\mathcal{B}_F=\mathcal{B}_{uF}.
\end{equation}
Therefore the involved probability measures have the same domain.\\

The paper is organized as follows: In the next section we learn about equivalent characterisations for weak convergence on $(\mathcal{F},\tau_{uF})$. Since the Fell topology is stronger than the upper Fell topology,
weak convergence on $(\mathcal{F},\tau_{F})$ entails this on $(\mathcal{F},\tau_{uF})$ and in general the reversal is not true, but it is under an additional assumption. Moreover, we find a close relationship to weak convergence
on $(\mathcal{F},\tau_{uV})$, where $\tau_{uV}$ is the \emph{upper Vietoris topology}. This is created when $\mathcal{K}$ is replaced by $\mathcal{F}$ in the construction of $\tau_{uF}$. A surprising result is that
every net of probability measures converges weakly on $(\mathcal{F},\tau_{uF})$. And whenever it converges to some $P$ it converges also to every $Q$, which \emph{dominates} $P$. Consequently, the space $(\Pi(\mathcal{F},\tau_{uF}), \tau_{weak})$ is compact and in general not Hausdorff. In section 3 we extend our results to \emph{random closed sets} in $E$ and \emph{measurable selections}. It is shown that the distributions of these selections converge weakly to a \emph{Choquet capacity} in the sense of Ferger \cite{Ferger1}, Definition 1.4. Under a uniqueness assumptions one obtains classical weak convergence in $(E,\mathcal{G})$.
Finally, in section 4 (appendix) we present some statements that are used in our proofs.

\section{Weak convergence of probability measures on the hyperspace $(\mathcal{F}, \tau_{uF})$}
By construction $\mathcal{S}_{uF}$ is a subbase of $\tau_{uF}$. This means that every basic open set
has the form $\bigcap_{i=1}^m \mathcal{M}(K_i)$ for compact sets $K_1,\ldots,K_m$ and $m \in \mathbb{N}$.
But since $\bigcap_{i=1}^m \mathcal{M}(K_i)=\mathcal{M}(\bigcup_{i=1}^m K_i)$ and $K:=\bigcup_{i=1}^m K_i$ is compact,
we see that $\mathcal{S}_{uF}$ actually is a base of $\tau_{uF}$. Consequently each open set $\textbf{O} \in \tau_{uF}$
has the representation
$$
 \textbf{O}=\bigcup_{K \in \mathcal{K}^*} \mathcal{M}(K)
$$
for some subclass $\mathcal{K}^* \subseteq \mathcal{K}$. Therefore, a general $\tau_{uF}$-closed set $\textbf{F}$ can be written as
$$
 \textbf{F}=\bigcap_{K \in \mathcal{K}^*} \mathcal{H}(K).
$$
Thus the equivalence $(\ref{wc}) \Leftrightarrow (\ref{closed})$ immediately yields a first characterisation:\\

\begin{proposition} \label{fc} The following two statements (i) and (ii) are equivalent:
\begin{itemize}
\item[(i)] $P_\alpha \rightarrow_w P$ on $(\mathcal{F},\tau_{uF})$.
\item[(ii)]
\begin{equation} \label{K}
\limsup_\alpha P_\alpha(\bigcap_{K \in \mathcal{K}^*} \mathcal{H}(K)) \le P(\bigcap_{K \in \mathcal{K}^*} \mathcal{H}(K))
\end{equation}
for every collection $\mathcal{K}^* \subseteq \mathcal{K}$ of compact sets in $E$.
\end{itemize}
\end{proposition}

We say that $(P_\alpha)$ is \emph{asymptotically compact-bounded}, if for each $\epsilon>0$ there exists a $K\in \mathcal{K}$ such that
\begin{equation} \label{ab}
 \liminf_\alpha P_\alpha(\{F \in \mathcal{F}:  F \subseteq K\}) \ge 1-\epsilon.
\end{equation}

Notice that $\{F \in \mathcal{F}:  F \subseteq K\}=(\mathcal{H}(K^c))^c$ is
$\tau_F$-closed, because $K^c \in \mathcal{G}$. Therefore it lies in $\mathcal{B}_F=\mathcal{B}_{uF}$, the domain of the $P_\alpha$.
Similarly, a single probability measure $P$ is said to be \emph{compact-bounded}, if for each $\epsilon>0$ there exists a $K \in \mathcal{K}$ such that
\begin{equation} \label{bP}
 P(\{F \in \mathcal{F}:  F \subseteq K\}) \ge 1-\epsilon.
\end{equation}

\begin{corollary} \label{hitF} If $P_\alpha \rightarrow_w P$ on $(\mathcal{F},\tau_{uF})$ and $(P_\alpha)$ is asymptotically compact-bounded, then
\begin{equation} \label{F}
\limsup_\alpha P_\alpha(\bigcap_{F \in \mathcal{F}^*} \mathcal{H}(F)) \le P(\bigcap_{F \in \mathcal{F}^*} \mathcal{H}(F))
\end{equation}
for every collection $\mathcal{F}^* \subseteq \mathcal{F}$ of closed sets in $E$.

If the limit $P$ is compact-bounded, then the reverse conclusion holds: (\ref{F}) implies that $P_\alpha \rightarrow_w P$ on $(\mathcal{F},\tau_{uF})$ and that $(P_\alpha)$ is asymptotically compact-bounded.
\end{corollary}

\begin{proof}
In Remark \ref{uV} below we will show that the intersections in (\ref{F}) are elements of $\mathcal{B}_{uF}$, so that all probabilities are well-defined. Let $\epsilon>0$ and $\textbf{F}:=\{F \in \mathcal{F}: F \subseteq K\}$ with $K$ as in (\ref{ab}). Then using the partition $\{\textbf{F},\textbf{F}^c\}$ one finds that $\bigcap_{F \in \mathcal{F}^*} \mathcal{H}(F)\subseteq \bigcap_{F \in \mathcal{F}^*} (\mathcal{H}(F) \cap \textbf{F}) \cup \textbf{F}^c$. Since $\mathcal{H}(F) \cap \textbf{F} \subseteq \mathcal{H}(F \cap K)$ for all $F \in \mathcal{F}^*$, we obtain:
\begin{equation} \label{FcapK}
 \limsup_\alpha P_\alpha(\bigcap_{F \in \mathcal{F}^*} \mathcal{H}(F))\le \limsup_\alpha P_\alpha(\bigcap_{F \in \mathcal{F}^*} \mathcal{H}(F\cap K))+\limsup_\alpha P_\alpha(\textbf{F}^c).
\end{equation}
Now, $F \cap K \in \mathcal{K}$, whence by Proposition \ref{fc}
\begin{equation} 
\limsup_\alpha P_\alpha(\bigcap_{F \in \mathcal{F}^*} \mathcal{H}(F\cap K)) \le P(\bigcap_{F \in \mathcal{F}^*} \mathcal{H}(F\cap K)) \le P(\bigcap_{F \in \mathcal{F}^*} \mathcal{H}(F)),
\end{equation}
where the last equality is trivial, because $F \cap K \subseteq F$ and so $\mathcal{H}(F\cap K)\subseteq \mathcal{H}(F)$ for all $F$.
By complementation the condition (\ref{ab}) is equivalent to
\begin{equation} \label{abc}
 \limsup_\alpha P_\alpha(\textbf{F}^c) \le \epsilon.
\end{equation}
Combining (\ref{FcapK})-(\ref{abc}) we arrive at
$$
 \limsup_\alpha P_\alpha(\bigcap_{F \in \mathcal{F}^*} \mathcal{H}(F))\le P(\bigcap_{F \in \mathcal{F}^*} \mathcal{H}(F))+\epsilon \quad \forall \; \epsilon>0.
$$
Taking the limit $\epsilon \rightarrow 0$ yields the assertion (\ref{F}).

As to the reverse implication let $\epsilon>0$ and $K \in \mathcal{K}$ as in (\ref{bP}). Fix some $r>0$ and consider the open $r$-neighborhood $K^{r-}$ and the closed $r$-neighborhood $K^r$ of $K$, see in the appendix for their definitions. Notice that $K \subseteq K^{r-} \subseteq K^r$. It follows:
\begin{eqnarray*}
& & \limsup_\alpha P_\alpha(F \in \mathcal{F}: F \nsubseteq K^r) = \limsup_\alpha P_\alpha(F \in \mathcal{F}: F \cap (K^r)^c \neq \emptyset)\\
&=& \limsup_\alpha P_\alpha(\mathcal{H}((K^r)^c))\\
&\le& \limsup_\alpha P_\alpha(\mathcal{H}((K^{r-})^c))   \quad \text{ because } K^r \supseteq K^{r-}\\
&\le& P(\mathcal{H}((K^{r-})^c)) \quad \text{ by (\ref{F}) with } \mathcal{F}^*=\{(K^{r-})^c\} \subseteq \mathcal{F}, \text{ because } K^{r-}\in \mathcal{G}\\
&\le& P(\mathcal{H}( K^c)) \quad \text{ because } K^{r-} \supseteq K\\
&=&  1-P(\mathcal{M}(K^c))=1-P(\{F \in \mathcal{F}:  F \subseteq K\}) \le \epsilon \quad \text{ by } (\ref{bP}).
\end{eqnarray*}
After complementation this shows asymptotic compact-boundedness of $(P_\alpha)$ upon noticing that $K^r$ is compact by Lemma \ref{Kr}. Since $\mathcal{K} \subseteq \mathcal{F}$, condition (\ref{F}) entails condition (\ref{K}).
Therefore an application of Proposition \ref{fc} yields that $P_\alpha \rightarrow_w P$ on $(\mathcal{F},\tau_{uF})$.
\end{proof}

\begin{remark} \label{uV} Let $\tau_{uV}$ be the \emph{upper Vietoris topology}. This is generated by $\mathcal{S}_{uV}:=\{\mathcal{M}(F): F \in \mathcal{F}\}$.
Notice that
\begin{equation} \label{HBinF}
\mathcal{H}(B) \in \mathcal{B}_F \text{ for all Borel-sets } B \in \mathcal{B}(E),
\end{equation}
confer Matheron \cite{Matheron}, p.30. Thus by Lemma 2.1.1 in Schneider and Weil \cite{Weil} and (\ref{HBinF}) the Borel-$\sigma$ algebra $\mathcal{B}_{uV}:=\sigma(\tau_{uV})$, like $\mathcal{B}_{uF}$, is the same as $\mathcal{B}_F$. Similarly as for $\mathcal{S}_{uF}$ one shows that $\mathcal{S}_{uV}$ is a base for $\tau_{uV}$, whence the family $\{\bigcap_{F \in \mathcal{F}^*} \mathcal{H}(F)), \mathcal{F}^* \subseteq \mathcal{F}\}$
is exactly the family of all $\tau_{uV}$-closed sets. In particular, these intersections are Borel-sets, i.e., they are elements of $\mathcal{B}_{uV}=\mathcal{B}_F=\mathcal{B}_{uF}$. As a further consequence we obtain that (\ref{F}) is equivalent to
$P_\alpha \rightarrow_w P$ on $(\mathcal{F},\tau_{uV})$.
\end{remark}

\vspace{0.3cm}
Recall that the family of sets occurring in (\ref{K}) coincides with the family of all $\tau_{uF}$-closed sets. Our next result shows that one can
reduce this family significantly. For its formulation we need the following denotation:
For a set $\textbf{A} \subseteq \mathcal{F}$ the boundary of $\textbf{A}$ with respect to the Fell-topology is denoted by $\partial_F\textbf{A}$.\\

\begin{theorem} \label{thm1} The following statements (i)-(iii) are equivalent:
\begin{itemize}
\item[(i)] $P_\alpha \rightarrow_w P$ on $(\mathcal{F},\tau_{uF})$.
\item[(ii)]
\begin{equation} \label{fK}
\limsup_\alpha P_\alpha(\bigcap_{i=1}^m \mathcal{H}(K_i)) \le P(\bigcap_{i=1}^m \mathcal{H}(K_i))
\end{equation}
for every $m \in \mathbb{N}$ and every finite collection $K_1,\ldots,K_m$ of non-empty compact sets in $E$.
\item[(iii)] The inequality (\ref{fK}) holds for every $m \in \mathbb{N}$ and every finite collection $K_1,\ldots,K_m$ of non-empty compact sets in $E$
such that $P(\partial_F \mathcal{H}(K_i))=0$ for all $1 \le i \le m$.

\end{itemize}
\end{theorem}

\begin{proof} The implication $(i) \Rightarrow (ii)$ follows from Proposition \ref{fc} with $\mathcal{K}^*=\{K_1,\ldots,K_m\}$.
For the reverse direction $(ii) \Rightarrow (i)$ one has to show (\ref{closed}), i.e.
\begin{equation} \label{uFc}
 \limsup_\alpha P_\alpha(\textbf{F}) \le P(\textbf{F}) \quad \text{for all } \tau_{uF}-\text{closed sets } \textbf{F}.
\end{equation}
We prove this by contradiction. So, assume that there exists a $\tau_{uF}$-closed set $\textbf{F}$ such that
\begin{equation} \label{contra}
 \limsup_\alpha P_\alpha(\textbf{F}) > P(\textbf{F}).
\end{equation}
The complement $\textbf{F}^c$ of $\textbf{F}$ in $\mathcal{F}$ lies in $\tau_{uF}$ and $\tau_{uF}$ has a countable base $\{\mathcal{M}(K): K \in \mathcal{K}_0 \}$, where $\mathcal{K}_0$ is a certain countable family of non-empty compact sets, confer, e.g., Ferger \cite{Ferger0} or Gersch \cite{Gersch}. Thus one can find a sequence $(K_i)_{i \in \mathbb{N}}$
of non-empty compact sets such that $\textbf{F}^c = \bigcup_{i \in \mathbb{N}} \mathcal{M}(K_i)$, whence $\textbf{F} = \bigcap_{i \in \mathbb{N}} \mathcal{H}(K_i)$. For every $k \in \mathbb{N}$ put $\textbf{F}_k := \bigcap_{i=1}^k \mathcal{H}(K_i)$. Then $\textbf{F}_k \downarrow \textbf{F}, k \rightarrow \infty$ and therefore
\begin{equation} \label{Fk}
 P(\textbf{F}_k) \downarrow P(\textbf{F}), k \rightarrow \infty.
\end{equation}
Let $a:= \frac{1}{2}(\limsup_\alpha P_\alpha(\textbf{F}) - P(\textbf{F})$. By assumption (\ref{contra}) the real number $a$ is positive and
by definition satisfies
\begin{equation} \label{contra2}
 \limsup_\alpha P_\alpha(\textbf{F}) = P(\textbf{F})+2 a > P(\textbf{F})+a.
\end{equation}
Recall that $(A, \le)$ is the directed set, where $\alpha$ is at home. By (\ref{Fk}) and $a>0$ there exists a $k \in \mathbb{N}$ such that
\begin{equation} \label{Fk2}
 P(\textbf{F}_k)-P(\textbf{F})< a/2.
\end{equation}
By definition $\limsup_\alpha P_\alpha(\textbf{F})= \inf_{\beta \in A} \sup_{\alpha \ge \beta} P_\alpha(\textbf{F})$ and thus (\ref{contra2}) ensures that
\begin{equation} \label{contra3}
 \sup_{\alpha \ge \beta} P_\alpha(\textbf{F}) > P(\textbf{F})+a \quad \forall \; \beta \in A.
\end{equation}
It follows that
\begin{equation}
\sup_{\alpha \ge \beta} P_\alpha(\textbf{F}_k)\ge \sup_{\alpha \ge \beta} P_\alpha(\textbf{F})>P(\textbf{F})+a/2+a/2>P(\textbf{F}_k)+a/2 \quad \forall \; \beta \in A.
\end{equation}
Here, the first inequality holds, because $\textbf{F}_k \supset \textbf{F}$, the second equality is (\ref{contra3}) and the last equality follows from (\ref{Fk2}). Taking the infimum over all $\beta \in A$ we obtain
\begin{eqnarray*}
 P(\textbf{F}_k)+a/2 &\le& \inf_{\beta \in A} \sup_{\alpha \ge \beta} P_\alpha(\textbf{F}_k)=\limsup_\alpha P_\alpha(\textbf{F}_k)=\limsup_\alpha P_\alpha(\bigcap_{i=1}^k \mathcal{H}(K_i))\\
 &\le& P(\bigcap_{i=1}^k \mathcal{H}(K_i))=P(\textbf{F}_k) \quad \text{ by assumption (ii)}.
\end{eqnarray*}
Consequently, $P(\textbf{F}_k)+a/2 \le P(\textbf{F}_k)$ and thus $a \le 0$ in contradiction to $a>0$.
Since the implication $(ii) \Rightarrow (iii)$ is trivial, it remains to prove that (iii) implies (ii).

For this purpose consider $R_i=R(K_i):=\{r>0: P(\partial \mathcal{H}(K_i^r))=0\}, 1 \le i \le m$. Here, $K_i^r$ is the closed $r$-neighborhood of $K_i$, which by Lemma \ref{Kr} in the appendix is compact.
We know from Lemma \ref{dense} in the appendix
that for each index $i$ the complement $R_i^c$ of $R_i$ is denumerable, whence $\bigcup_{i=1}^m R_i^c$ is denumerable as well.
As a consequence $R(K_1,\ldots,K_m):= \bigcap_{i=1}^m R(K_i)=(\bigcup_{i=1}^m R_i^c)^c$ lies dense in $[0,\infty)$.
Thus there exists a sequence $(r_j)_{j \in \mathbb{N}}$ in $R(K_1,\ldots,K_m)$ such that $r_j \downarrow 0, j \rightarrow \infty$.
Conclude that
\begin{equation} \label{ass3}
 \limsup_\alpha P_\alpha(\bigcap_{i=1}^m \mathcal{H}(K_i)) \le \limsup_\alpha P_\alpha(\bigcap_{i=1}^m \mathcal{H}(K_i^{r_j})) \le P(\bigcap_{i=1}^m \mathcal{H}(K_i^{r_j})) \quad \forall \; j \in \mathbb{N}.
\end{equation}
Here, the first inequality holds, because $K_i \subseteq K_i^{r_j}$ and therefore $\mathcal{H}(K_i) \subseteq \mathcal{H}(K_i^{r_j})$.
As to the second inequality observe that $r_j \in R(K_i)$ for every $1 \le i \le m$, which means that $P(\partial \mathcal{H}(K_i^{r_j})=0$
for all $i$ and so we can use assumption (iii) taking into account that the $K_i^{r_j}$ are all compact. Finally, consider $\textbf{E}_j := \bigcap_{i=1}^m \mathcal{H}(K_i^{r_j}), j \in \mathbb{N}$,
which by monotonicity of $(r_j)$ are monotone decreasing. So $\sigma$-continuity of $P$ from above yields:
\begin{equation} \label{Cij1}
\lim_{j \rightarrow \infty} P(\bigcap_{i=1}^m \mathcal{H}(K_i^{r_j}))= \lim_{j \rightarrow \infty} P(\textbf{E}_j)=P(\bigcap_{j \in \mathbb{N}} \textbf{E}_j)=P(\bigcap_{i=1}^m \bigcap_{j \in \mathbb{N}}\mathcal{H}(K_i^{r_j})).
\end{equation}
Put $C_{ij}:=K_i^{r_j}$. Observe that for each fixed $1 \le i \le m$ we have that
$C_{ij} \downarrow K_i, j \rightarrow \infty$ and that $(C_{ij})_{j \in \mathbb{N}} \subseteq \mathcal{K}$. Thus $\bigcap_{j \in \mathbb{N}}\mathcal{H}(C_{ij}) = \mathcal{H}(K_i)$ for every $1 \le i \le m$ by Lemma \ref{Heq} in the appendix. Infer that
\begin{equation} \label{Cij2}
 P(\bigcap_{i=1}^m \bigcap_{j \in \mathbb{N}}\mathcal{H}(K_i^{r_j}))=P(\bigcap_{i=1}^m \mathcal{H}(K_i)).
\end{equation}
Thus taking the limit $j \rightarrow \infty$ in (\ref{ass3}) yields (ii) by (\ref{Cij1}) and (\ref{Cij2}).
\end{proof}

Next we relate $P_\alpha \rightarrow_w P$ on $(\mathcal{F}, \tau_{uF})$ with $P_\alpha \rightarrow_w P$ on $(\mathcal{F}, \tau_{F})$.  A first simple relation is given in:\\

\begin{proposition} \label{p1} If $P_\alpha \rightarrow_w P$ on $(\mathcal{F}, \tau_{F})$, then $P_\alpha \rightarrow_w P$ on $(\mathcal{F}, \tau_{uF})$.
\end{proposition}

\begin{proof} This follows from the equivalent characterization (\ref{open}) taking into account that $\tau_F \supseteq \tau_{uF}$.
\end{proof}

Weak convergence on $(\mathcal{F}, \tau_{F})$ is well-studied in contrast to that on $(\mathcal{F}, \tau_{uF})$. For example
$P_\alpha \rightarrow_w P$ on $(\mathcal{F}, \tau_{F})$ if and only if $P_\alpha(\mathcal{H}(K)) \rightarrow P(\mathcal{H}(K))$ for all
compact $K$ with $P(\partial_F \mathcal{H}(K))=0$. This and other characterisations can be found in in Molchanov \cite{Molchanov}.
Here, only sequences of probability measures are considered. However, in Ferger \cite{Ferger0} we carry over the theory to nets of probability measures.

Our next result is really astonishing and has interesting consequences. Here we say that $Q$ \emph{dominates} $P$ (in symbol: $P \le Q$) if
$P(\bigcap_{i=1}^m \mathcal{H}(K_i))\le Q(\bigcap_{i=1}^m \mathcal{H}(K_i))$ for every $m \in \mathbb{N}$ and for every collection $K_1,\ldots,K_m$ of non-empty compact sets.\\

\begin{lemma} \label{allconverge} Every net $(P_\alpha)$ is weakly convergent on $(\mathcal{F}, \tau_{uF})$ with limit $\delta_E$, the Dirac-measure at point $E$.
Moreover, if $P_\alpha \rightarrow_w P$ on $(\mathcal{F}, \tau_{uF})$, then $P_\alpha \rightarrow_w Q$ on $(\mathcal{F}, \tau_{uF})$ for each $Q$ that dominates $P$.
\end{lemma}

\begin{proof} If $K_1,\ldots,K_m$ are non-empty compact sets, then $\delta_E(\bigcap_{i=1}^m \mathcal{H}(K_i))=1$, whence (ii) of Theorem \ref{thm1} is fulfilled
and therefore $P_\alpha \rightarrow_w \delta_E$ on $(\mathcal{F},\tau_{uF})$.
Each $Q$ that dominates $P$ satisfies (ii) of Theorem \ref{thm1}, which shows the second claim by another application of Theorem \ref{thm1}.
\end{proof}

As immediate consequences we obtain:\\

\begin{corollary} \label{properties} The topological space $(\Pi(\mathcal{F}, \tau_{uF}), \tau_{weak})$ is compact.
In general it is not Hausdorff and therefore not metrizable.
\end{corollary}

\vspace{0.4cm}
\begin{corollary} In general the reverse conclusion in Proposition \ref{p1} is not true.
\end{corollary}

\vspace{0.4cm}
In view of the last result, the question arises under which additional conditions the reversal applies.
The answer involves the family $\textbf{F}_{0,1}$ of all sets with at most one element,
i.e., $\textbf{F}_{0,1}=\{\emptyset\} \cup \{\{x\}: x \in E\}$. Since the empty set and all singletons are closed,
$\textbf{F}_{0,1} \subseteq \mathcal{F}$. According to Lemma \ref{F01} in the appendix $\textbf{F}_{0,1}$ is $\tau_F$-closed
and thus in particular is a Borel-set: $\textbf{F}_{0,1} \in \mathcal{B}_F$. \\

\begin{theorem} \label{thm2} Suppose that
\begin{itemize}
\item[(i)] $P_\alpha \rightarrow_w P$ on $(\mathcal{F}, \tau_{uF})$,
\item[(ii)] For each $\epsilon>0$ there exists a $K \in \mathcal{K}$ such that
\begin{equation} \label{ii}
 \liminf_\alpha P_\alpha(\{F \in \mathcal{F}: \emptyset \neq F \subseteq K\}) \ge 1-\epsilon,
\end{equation}
(This is a bit more than asymptotic compact-boundedness (\ref{ab}).)
\item[(iii)] $P(\textbf{F}_{0,1}) = 1$.
\end{itemize}
Then
\begin{equation} \label{result}
 P_\alpha \rightarrow_w P \text{ on } (\mathcal{F}, \tau_{F}).
\end{equation}
\end{theorem}

\begin{proof} Let $\textbf{F}$ be $\tau_F$-closed and for each $\epsilon>0$ let $K \in \mathcal{K}$ as in (ii). Put
$\textbf{B}:=\{F \in \mathcal{F}: \emptyset \neq F \subseteq K\}$. Since $\textbf{B}= \mathcal{H}(E) \cap \mathcal{M}(K^c)$ Lemma 2.1.1 in
Schneider and Weil \cite{Weil} ensures that $\textbf{B} \in \mathcal{B}_F$ is a Borel-set and hence the probabilities in
(\ref{ii}) are well-defined. The decomposition $\textbf{F}=(\textbf{F} \cap \textbf{B}) \cup (\textbf{F} \cap \textbf{B}^c)$ yields that
$\textbf{F} \subseteq \textbf{C} \cup \textbf{B}^c$ with $\textbf{C}=\textbf{F} \cap \textbf{B} \in \mathcal{B}_F$.
According to (ii) $\liminf_\alpha P_\alpha (\textbf{B}) \ge 1-\epsilon$, whence by complementation $\limsup_\alpha P_\alpha (\textbf{B}^c) \le \epsilon$. It follows that
\begin{equation} \label{i1}
\limsup_\alpha P_\alpha(\textbf{F}) \le \limsup_\alpha P_\alpha (\textbf{C}) + \epsilon \quad \text{ for every } \epsilon>0.
\end{equation}
In the sequel $cl_{uF}$ and $cl_F$ refer to the closure with respect to the upper Fell-topology and the Fell-topology, respectively.
The next relation is the key of our proof.

\begin{equation} \label{key}
 \textbf{C} \subseteq cl_{uF}(\textbf{C}) \subseteq cl_F(\textbf{C}) \cup \textbf{F}_{0,1}^c.
\end{equation}
The first $\subseteq$ holds by definition of the closure. As to the second one let $F \in cl_{uF}(\textbf{C})$. Then there exists a net
$(F_\alpha)$ in $\textbf{C}$ with $F_\alpha \rightarrow F$ in $(\mathcal{F},\tau_{uF})$. If $F \notin \textbf{F}_{0,1}$, then it lies in the
set on the right side of (\ref{key}) as desired. So it remains to consider $F \in \textbf{F}_{0,1}$. If $F=\emptyset$, then
$F_\alpha \rightarrow F$ in $(\mathcal{F},\tau_{F})$ by Lemma \ref{reverse} (b) in the appendix and consequently $F \in cl_F(\textbf{C})$.
If $F=\{x\}$ is a singleton, then we use that $\emptyset \neq F_\alpha \subseteq K$ for all $\alpha \in A$, because $F_\alpha \in \textbf{C} \subseteq \textbf{B}$. Thus we can apply Lemma \ref{reverse} (a) in the appendix, which yields that $F_\alpha \rightarrow F$ in $(\mathcal{F},\tau_{F})$, whence $F \in cl_F(\textbf{C})$ also in that last case. This finally shows that (\ref{key}) is true.

Next, observe that
\begin{equation} \label{i2}
\limsup_\alpha P_\alpha(\textbf{C}) \le \limsup_\alpha P_\alpha(cl_F(\textbf{C})) \le P(cl_{uF}(\textbf{C}),
\end{equation}
by the first part of (\ref{key}) and by assumption (i), because $cl_{uF}(\textbf{C})$ is $\tau_{uF}$-closed. The second part of (\ref{key}) gives
\begin{equation} \label{i3}
 P(cl_{uF}(\textbf{C}))\le P(cl_{F}(\textbf{C}))+P(\textbf{F}_{0,1}^c)=P(cl_{F}(\textbf{C}))\le P(cl_{F}(\textbf{F}))=P(\textbf{F}).
\end{equation}
Here, the first inequality holds by the second part of (\ref{key}), the equality holds by assumption (iii), the second inequality holds by $cl_F(\textbf{C}) \subseteq cl_F(\textbf{F})$, because $\textbf{C} \subseteq \textbf{F}$ and the last equality holds, since $\textbf{F}$ is $\tau_F$-closed. Combining (\ref{i1}), (\ref{i2}) and (\ref{i3}) results in
$$
\limsup_\alpha P_\alpha(\textbf{F}) \le P(\textbf{F})+ \epsilon \quad \forall \; \epsilon>0.
$$
Taking the limit $\epsilon \rightarrow 0$ yields the convergence in (\ref{result}).
\end{proof}

\section{Distributional convergence of random closed sets in $(\mathcal{F},\tau_{uF})$}
Let $(\Omega,\mathcal{A},\mathbb{P})$ be a probability space. A map $C:\Omega \rightarrow \mathcal{F}$ is called \emph{random closed set (in $E$ on} $(\Omega,\mathcal{A},\mathbb{P})$), if it is
$\mathcal{A}-\mathcal{B}_F$ measurable. Its distribution $\mathbb{P}\circ C^{-1}$ is a probability measure on $(\mathcal{F},\tau_F)$, but also on $(\mathcal{F},\tau_{uF})$ by (\ref{Borel}). Conversely, by the canonical construction every probability measure $P$ on $(\mathcal{F},\tau_F)$ or $(\mathcal{F},\tau_{uF})$, respectively, is the
distribution of a random closed set on a probability space $(\Omega,\mathcal{A},\mathbb{P})$. Indeed, one can take $(\Omega,\mathcal{A},\mathbb{P})=
(\mathcal{F},\mathcal{B}_F,P)=(\mathcal{F},\mathcal{B}_{uF},P)$ and $C$ is equal to the identity map.

If $(C_\alpha)_{\alpha \in A}$ is a net of random closed sets $C_\alpha$ in $E$ on $(\Omega_\alpha,\mathcal{A}_\alpha,\mathbb{P}_\alpha)$, then as usual we define distributional convergence by weak convergence of the distributions.
More precisely, $C_\alpha \stackrel{\mathcal{D}}{\rightarrow} C$ in $(\mathcal{F},\tau_{uF})$, if $\mathbb{P}_\alpha \circ C_\alpha^{-1} \rightarrow_w \mathbb{P}\circ C^{-1}$ on $(\mathcal{F},\tau_{uF})$ and $C_\alpha \stackrel{\mathcal{D}}{\rightarrow} C$ in $(\mathcal{F},\tau_{F})$, if $\mathbb{P}_\alpha \circ C_\alpha^{-1} \rightarrow_w \mathbb{P}\circ C^{-1}$ on $(\mathcal{F},\tau_{F})$.

The short discussion above shows that every result in the last section (except for Corollary \ref{properties}) can be formulated in terms of random closed sets. For example Theorem \ref{thm1} takes the following form:\\

\begin{theorem} \label{thm1rcs} The following statements are equivalent:
\begin{itemize}
\item[(i)] $C_\alpha \stackrel{\mathcal{D}}{\rightarrow} C$ in $(\mathcal{F},\tau_{uF})$.
\item[(ii)]
\begin{equation} \label{fKrcs}
\limsup_\alpha \mathbb{P}_\alpha(C_\alpha \cap K_1\neq \emptyset,\ldots, C_\alpha \cap K_m \neq \emptyset) \le \mathbb{P}(C \cap K_1\neq \emptyset,\ldots, C \cap K_m \neq \emptyset)
\end{equation}
for every $m \in \mathbb{N}$ and every finite collection $K_1,\ldots,K_m$ of non-empty compact sets in $E$.
\item[(iii)] The inequality (\ref{fKrcs}) holds for every $m \in \mathbb{N}$ and every finite collection $K_1,\ldots,K_m$ of non-empty compact sets in $E$
such that $\mathbb{P}(C \in \partial_F \mathcal{H}(K_i))=0$ for all $1 \le i \le m$.
\end{itemize}
\end{theorem}

Notice that by Lemma \ref{boundary} (i) $\mathbb{P}(C \in \partial_F \mathcal{H}(K_i))=\mathbb{P}(C \cap K_i \neq \emptyset, C \cap K_i^0 = \emptyset)$. Thus we see that Theorem \ref{thm1rcs} is a generalization of Vogel's \cite{Vogel} Lemma 2.1, where only sequences rather than nets of random closed sets are considered and furthermore $E$ is required to be the euclidian space $\mathbb{R}^d$.

Assume $C$ and $D$ are random closed sets with $\mathbb{P}^*(C \nsubseteq D)=0$, where $\mathbb{P}^*$ is the outer measure of $\mathbb{P}$.
This means that $\{C \nsubseteq D\}$ is a $\mathbb{P}$-null set and by completion of the probability space $(\Omega,\mathcal{A},\mathbb{P})$ we can achieve that $\{C \nsubseteq D\} \in \mathcal{A}$. Now we can say that $C \subseteq D \; \mathbb{P}$-almost surely (a.s.). In this case
$\mathbb{P} \circ D^{-1}$ dominates $\mathbb{P} \circ C^{-1}$ and from Lemma \ref{allconverge} we can deduce:\\

\begin{lemma} \label{sup}
$C_\alpha \stackrel{\mathcal{D}}{\rightarrow} E$ in $(\mathcal{F},\tau_{uF})$ for all nets $(C_\alpha)$ of random closed sets in $E$.
Moreover, if $C_\alpha \stackrel{\mathcal{D}}{\rightarrow} C$ in $(\mathcal{F},\tau_{uF})$ and $D$ is a random closed set with
$C \subseteq D \; \mathbb{P}$-a.s., then $C_\alpha \stackrel{\mathcal{D}}{\rightarrow} D$ in $(\mathcal{F},\tau_{uF})$.
\end{lemma}

\vspace{0.3cm}
In short, every superset of a limit set is also a limit set. Conversely,
every net $(D_\alpha)$ of subsets, i.e., $D_\alpha \subseteq C_\alpha \; \mathbb{P}_\alpha$-a.s., also converges to $C$.
In fact, a somewhat more general result applies:\\

\begin{lemma} \label{sub} Let $(C_\alpha)$ and $(D_\alpha)$ be nets of random closed sets in $E$ on $(\Omega_\alpha,\mathcal{A}_\alpha,\mathbb{P}_\alpha)$ such that
\begin{equation} \label{subsets}
 \limsup_\alpha \mathbb{P}_\alpha(D_\alpha \nsubseteq C_\alpha)=0.
\end{equation}
Then $C_\alpha \stackrel{\mathcal{D}}{\rightarrow} C$ in $(\mathcal{F},\tau_{uF})$ entails $D_\alpha \stackrel{\mathcal{D}}{\rightarrow} C$ in $(\mathcal{F},\tau_{uF})$.
\end{lemma}

\begin{proof} Let $K_1,\ldots,K_m$ be non-empty compact sets in $E$. Then
\begin{eqnarray*}
 & &\bigcap_{i=1}^m \{D_\alpha \cap K_i \neq \emptyset\}\\
 &=&\bigcap_{i=1}^m (\{D_\alpha \cap K_i \neq \emptyset\}\cap \{D_\alpha \subseteq C_\alpha\}) \cup (\bigcap_{i=1}^m \{D_\alpha \cap K_i \neq \emptyset\}) \cap \{D_\alpha \nsubseteq C_\alpha\}\\
 &\subseteq& \bigcap_{i=1}^m \{C_\alpha \cap K_i \neq \emptyset\} \cup \{D_\alpha \nsubseteq C_\alpha\}.
\end{eqnarray*}
Consequently,
\begin{eqnarray*}
 & &\limsup_\alpha \mathbb{P}_\alpha(\bigcap_{i=1}^m \{D_\alpha \cap K_i \neq \emptyset\})\\
  &\le& \limsup_\alpha \mathbb{P}_\alpha(\bigcap_{i=1}^m \{C_\alpha \cap K_i \neq \emptyset\})+\limsup_\alpha \mathbb{P}_\alpha(D_\alpha \nsubseteq C_\alpha)\\
 &=&\limsup_\alpha \mathbb{P}_\alpha(\bigcap_{i=1}^m \{C_\alpha \cap K_i \neq \emptyset\}) \quad \text{ by } (\ref{subsets})\\
 &\le& \mathbb{P}(C \cap K_1\neq \emptyset,\ldots, C \cap K_m \neq \emptyset) \quad \text{  by Theorem \ref{thm1rcs}}.
\end{eqnarray*}
Another application of Theorem \ref{thm1rcs} yields the assertion.
\end{proof}

With our next result we can give an answer to the following question: If $C_\alpha \stackrel{\mathcal{D}}{\rightarrow} C$ in $(\mathcal{F},\tau_{uF})$, then what can we say about the asymptotic behaviour of random variables $\xi_\alpha \in C_\alpha$?\\

\begin{theorem} \label{thm3} For each $\alpha \in A$ let $\xi_\alpha:(\Omega_\alpha,\mathcal{A}_\alpha,\mathbb{P}_\alpha) \rightarrow (E,\mathcal{B}(E))$ be a measurable map (random variable in $E$). Suppose that:
\begin{itemize}
\item[(i)] $C_\alpha \stackrel{\mathcal{D}}{\rightarrow} C$ in $(\mathcal{F},\tau_{uF})$.
\item[(ii)] $\limsup_\alpha \mathbb{P}_\alpha(\xi_\alpha \notin C_\alpha)=0$.
\item[(iii)] For every $\epsilon>0$ there exists a $K \in \mathcal{K}$ such that
$$\liminf_\alpha \mathbb{P}_\alpha(\xi_\alpha \in K) \ge 1-\epsilon.$$
Then
\begin{equation} \label{T}
 \limsup_\alpha \mathbb{P}_\alpha(\xi_\alpha \in F) \le T_C(F) \quad \forall \; F \in \mathcal{F},
\end{equation}
where $T_C$ is a Choquet-capacity, namely the capacity functional of $C$ given by $T_C(B)=\mathbb{P}(C \cap B \neq \emptyset), B \in \mathcal{B}(E)$.

If in addition $C \subseteq \{\xi\}\; \mathbb{P}$-a.s. for some random variable $\xi$ in $E$ on $(\Omega,\mathcal{A},\mathbb{P})$, then
\begin{equation} \label{dcxi}
 \xi_\alpha \stackrel{\mathcal{D}}{\rightarrow} \xi \quad \text{ in } (E,\mathcal{G}).
\end{equation}
\end{itemize}
\end{theorem}

\begin{proof} By Lemma 2.1.1 in Schneider and Weil \cite{Weil} $\mathcal{B}_F=\sigma(\{\mathcal{H}(G):G\in \mathcal{G}\})$. Thus $D_\alpha :=\{\xi_\alpha\}$ are random closed sets, because $\{D_\alpha \in \mathcal{H}(G)\}=\{\xi_\alpha \in G\} \in \mathcal{A}_\alpha$.
From (i) and (ii) it follows by Lemma \ref{sub} that $\{\xi_\alpha\} \stackrel{\mathcal{D}}{\rightarrow} C$ in $(\mathcal{F},\tau_{uF}).$
Conclude with (iii) and Corollary \ref{hitF} (in the formulation for random closed sets and with $\mathcal{F}^*=\{F\}$ a singleton) that
\begin{equation} \label{TT}
 \limsup_\alpha \mathbb{P}_\alpha(\xi_\alpha \in F) = \limsup_\alpha \mathbb{P}_\alpha(\{\xi_\alpha\} \cap F \neq \emptyset) \le \mathbb{P}(C\cap F \neq \emptyset)=T_C(F) \quad \forall \; F \in \mathcal{F}.
\end{equation}
This shows (\ref{T}). Under the additional assumption $C \subseteq \{\xi\}$ a.s. it follows that $T_C(F) \le \mathbb{P}(\xi \in F)$, whence we can infer from
(\ref{TT}) that $\limsup_\alpha \mathbb{P}_\alpha(\xi_\alpha \in F)\le \mathbb{P}(\xi \in F)$ for all closed $F$, which by (\ref{closed}) yields
the distributional convergence (\ref{dcxi}).
\end{proof}

If a net $(\xi_\alpha)$ satisfies (\ref{T}), then we say that it \emph{converges in distribution to (the random closed set)} $C$ and denote this by $$\xi_\alpha \stackrel{\mathcal{D}}{\rightarrow} C.$$ This new type of distributional convergence has been introduced and analyzed in Ferger \cite{Ferger1}.
Here, for instance we show in Theorem 4.2 that $\xi_\alpha \stackrel{\mathcal{D}}{\rightarrow} C$ is equivalent to $\{\xi_\alpha\}\stackrel{\mathcal{D}}{\rightarrow} C$ in $(\mathcal{F}, \tau_{uV})$.\\

\begin{remark} Suppose that $C \subseteq \{\xi\}\; \mathbb{P}$-a.s. Under the assumptions (i)-(iii) of the last theorem it follows that actually
$C = \{\xi\}\; \mathbb{P}$-a.s. To see this notice that $E \in \mathcal{F}$ and $\mathbb{P}_\alpha(\xi_\alpha \in E)=1$ for all $\alpha \in A$.  Therefore (\ref{T}) with $F=E$  yields that $\mathbb{P}(C \cap E \neq \emptyset)=\mathbb{P}(C \neq \emptyset)=1$, whence $\mathbb{P}(C = \emptyset)=0$, which in turn by the assumption on $C$ gives the a.s. equality.
\end{remark}

\vspace{0.3cm}
\begin{remark} Since $T_C(B)=\mathbb{P}(C \in \mathcal{H}(B))$, the capacity functional $T_C$
is by (\ref{HBinF}) well-defined on the Borel-$\sigma$ algebra $\mathcal{B}(E)$ on $(E,\mathcal{G})$. In general, $T_C$
is not a probability measure. In fact, it is a probability measure if and only if there exists a random
variable $\{\xi\}$ such that $C \stackrel{\mathcal{D}}{=} \{\xi\}$, see Ferger \cite{Ferger0}.
For further properties of $T_C$ we refer to Molchanov \cite{Molchanov}.
\end{remark}

\vspace{0.3cm}
If $C_\alpha$ is a random closed set and $\xi_\alpha$ is a random variable in $E$ with $\xi_\alpha \in C_\alpha\; \mathbb{P}_\alpha$-a.s., then $\xi_\alpha$ is called a \emph{measurable selection of} $C_\alpha$.
By the \emph{Fundamental selection theorem}, confer Molchanov \cite{Molchanov} on p.77, the existence of $\xi_\alpha$ is guaranteed.
A net $(\xi_\alpha)$ satisfying condition (iii) of the above Theorem \ref{thm3} is called \emph{asymptotically tight}. This condition
is much weaker than the classical (uniform) tightness, which requires that $\mathbb{P}_\alpha(\xi_\alpha \notin K)\le \epsilon$
for all $\alpha \in A$ and not only in the limit.

The following corollary provides an answer to the question posed above.\\

\begin{corollary} \label{corthm3} Assume that $C_\alpha \stackrel{\mathcal{D}}{\rightarrow} C$ in $(\mathcal{F},\tau_{uF})$ and that  $(\xi_\alpha)$ is a net of measurable selections $\xi_\alpha$ of $C_\alpha$.
If $(\xi_\alpha)$ is asymptotically tight, then $\xi_\alpha \stackrel{\mathcal{D}}{\rightarrow} C$.
In case that $C \subseteq \{\xi\}$ a.s. for some random variable $\xi$ we obtain:
$\xi_\alpha \stackrel{\mathcal{D}}{\rightarrow} \xi$ in $(E,\mathcal{G})$.
\end{corollary}

\begin{proof} Conditions (i) and (iii) of Theorem \ref{thm3} are fulfilled by assumption. Since each $\xi_\alpha$ is a measurable selection, we have that $\mathbb{P}(\xi_\alpha \notin C_\alpha)=0$ for all $\alpha \in A$, whence condition (ii) is trivially fulfilled and thus Theorem \ref{thm3} yields the assertion.
\end{proof}

\section{Appendix}
In this section we present several results, which we use in our proofs above. For some of these, the statements are known in case
$E$ is a finite-dimensional linear space with a metric $d$ such as for example $E=\mathbb{R}^d$. More details are given in our notes at the
end of the appendix.

Since $(E,\mathcal{G})$ is locally compact, second-countable and Hausdorff it is metrizable. By Theorem 2 of Vaughan \cite{Vaughan} the underlying metric $d$ can be chosen such that:

\begin{equation} \label{rc}
 \text{Every bounded set is relatively compact.}
\end{equation}

In addition $(E,d)$ is complete and thus a polish metric space. For the extremely useful result (\ref{rc}) confer also Engelking \cite{Engel}, Exercise 4.2C on p. 265.
Although Vaughan's theorem was published in 1937, it does not seem to be so well known.

Given a point $x \in E$ and a non-empty subset $A \subseteq E$ let $d(x,A):= \inf\{d(x,a): a \in A\}$ denote the distance of $x$ from $A$.
As usual $B(x,r):=\{y \in E: d(y,x) <r\}$ denotes the (open) ball with center $x$ and radius $r>0$. Moreover, $A^r:=\{x \in E: d(x,A)\le r\}$ and $A^{r-}:=\{x \in E: d(x,A)< r\}$
are respectively the \emph{closed} and \emph{open} $r$-\emph{neighborhoods} of A, where $r>0$.

We use the usual notation $Int(A)\equiv A^o, cl(A)\equiv \overline{A}$ and $\partial A$ for the interior, the closure and boundary, respectively, of $A$ in $(E,\mathcal{G})$.\\

\begin{lemma} \label{Kr} If $K \neq \emptyset$ is compact, then $K^r$ is compact for all $r > 0$.
\end{lemma}

\begin{proof} Firstly, observe that $K^r = d(\cdot,K)^{-1}((-\infty,r])$ is closed as the pre-image of the closed half-line $(-\infty,x]$ under the continuous function $x \mapsto d(x,K)$.
Therefore by (\ref{rc}) it suffices to show that $K^r$ is bounded, i.e., there exist some $x_0 \in E$ and some $s>0$ such that $K^r \subseteq B(x_0,s)$.
We prove this by contradiction. So, assume that $K^r \nsubseteq B(x_0,s)$ for each $x_0 \in E$ and for all $s>0$, that means there exists
some $y \in E$ with $d(y,K) \le r$, but $d(y,x_0) \ge s$. Since $K \in \mathcal{K}$, there exists some $z \in K$ such that $d(y,K)=d(y,z)$.
By the triangle-inequality we know that $d(x_0,y) \le d(x_0,z)+d(z,y)$, which implies that
$s-r \le d(x_0,y)- d(z,y) \le d(x_0,z)$. It follows that $s-r \le d(x_0,z)$ for all $x_0 \in E$ and for all $s>0$. Choosing $x_0=z$ leads to
$s \le r$ for all $s>0$, which is a contradiction. Now, by closedness $K^r = \overline{K^r}$ and $\overline{K^r} \in \mathcal{K}$ by (\ref{rc}),
whence $K^r$ is compact.
\end{proof}

For $\textbf{A} \subseteq \mathcal{F}$ let
$Int_F(\textbf{A})$ and $cl_F(\textbf{A})$ denote the interior and closure, respectively, of $\textbf{A}$ in $(\mathcal{F}, \tau_F)$.
In the following we will use that by construction of $\tau_F$ the basic open sets $\textbf{B}$ are all of the type
\begin{equation} \label{bos}
\textbf{B}= \mathcal{M}(K) \cap \mathcal{H}(G_1) \cap \ldots \cap \mathcal{H}(G_l)
\end{equation}
with $K \in \mathcal{K}, G_1, \ldots, G_l \in \mathcal{G}$ and
$l \in \mathbb{N}_0$. For $l=0$ we obtain $\textbf{B}=\mathcal{M}(K)$.\\

\begin{lemma} \label{intF} If $A \subseteq E$ is an arbitrary subset,
then
\begin{equation} \label{IntHA}
 Int_F(\mathcal{H}(A))=\mathcal{H}(Int(A)).
\end{equation}
\end{lemma}

\begin{proof} W.l.o.g. $A$ is nonempty, because otherwise equation (\ref{IntHA}) is trivially fulfilled. We first prove the relation $\subseteq$. So, let $F \in Int_F(\mathcal{H}(A))$.
By definition of the interior there exists an $\textbf{O} \in \tau_F$ such that $F \in \textbf{O}$ and $\textbf{O} \subseteq \mathcal{H}(A)$.
Since every open set is the union of basic open sets we find a basic open set as in (\ref{bos}) such that
\begin{equation} \label{inHA}
 F \in \mathcal{M}(K) \cap \mathcal{H}(G_1) \cap \ldots \cap \mathcal{H}(G_l) \subseteq \mathcal{H}(A).
\end{equation}
In particular, $F \in \mathcal{H}(A)$, i.e., $F \cap A \neq \emptyset$, whence $F \neq \emptyset$, which is the same as $F \in \mathcal{H}(E)$.
Now, $E \in \mathcal{G}$ and therefore we can assume that $l \ge 1$. It follows that:
\begin{equation} \label{exi}
 \exists \; i \in \{1, \ldots,l\} : G_i \cap  K^c \subseteq A.
\end{equation}
We prove (\ref{exi}) by contradiction. For that purpose recall that, if $U$ and $V$ are subsets of $E$, then the following equivalence holds: $U \subseteq V \Leftrightarrow U \cap V^c = \emptyset$. So let us assume that (\ref{exi}) is not true, which means that
$G_i \cap K^c \cap A^c \neq \emptyset \quad \forall \; i \in \{1, \ldots,l\}.$ Consequently, for each $i \in \{1, \ldots,l\}$ there exists a point
$x_i$ with $x_i \in G_i, x_i \notin K$ and $x_i \notin A$. Introduce $H:=\{x_1, \ldots, x_l\}$. Then $H$ has the following properties:
$H \in \mathcal{F}, H \cap K = \emptyset$ (because all $x_i$ are not $K$) and $H \cap G_i \neq \emptyset \; \forall \; 1 \le i \le l$ (because
$x_i \in H \cap G_i$ for each $i$). Thus
$$
 H \in \mathcal{M}(K) \cap \mathcal{H}(G_1) \cap \ldots \mathcal{H}(G_l) \subseteq \mathcal{H}(A),
$$
where the last relation $\subseteq$ holds by (\ref{inHA}). It follows that $H \in \mathcal{H}(A)$, i.e., $H \cap A \neq \emptyset$ in contradiction to $x_i \notin A$ for all $i$. This is the proof of (\ref{exi}).

Put $G:=G_i \cap K^c$ with $G_i$ from (\ref{exi}). Then $G \in \mathcal{G}$ and $G \subseteq A$. Moreover:
\begin{equation} \label{FinHG}
\emptyset \neq F \cap G_i=(F \cap G_i \cap K^c) \cup (F \cap G_i \cap K)=F \cap G_i \cap K^c= F \cap G.
\end{equation}
Here, the first relation $\neq$ follows from $F \in \mathcal{H}(G_i)$ by (\ref{inHA}). The subsequent equality follows from the decomposition $E=K^c \cup K$. The next equality also follows from (\ref{inHA}), which implies that $F \in \mathcal{M}(K)$. Hence $F \cap K = \emptyset$ and so
$F \cap G_i \cap K = \emptyset$ a fortiori. Since $G \in \mathcal{G}$ and $G \subseteq A$, we have that $G \subseteq Int(A)$.
Deduce from (\ref{FinHG}) that $\emptyset \neq F \cap G \subseteq F \cap Int(A)$, whence $F \cap Int(A) \neq \emptyset$, which is the same as
$F \in \mathcal{H}(Int(A))$.

To see the reverse relation $\supseteq$ in (\ref{IntHA}) assume that $F \in \mathcal{H}(Int(A))$. In the sequel we use the following property of $\mathcal{H}(\cdot)$:
$\mathcal{H}(\bigcup_{A \in \mathcal{A}} A)= \bigcup_{A \in \mathcal{A}}\mathcal{H}(A)$ for every family $\mathcal{A}$ of subsets of $E$.
Apply this property to the family $\mathcal{A}=\{G \in \mathcal{G}: G \subseteq A\}$. It yields that
$$
 \mathcal{H}(Int(A)) = \mathcal{H}(\bigcup_{G \in \mathcal{A}} G) = \bigcup_{G \in \mathcal{A}} \mathcal{H}(G).
$$
Thus one finds an open $G$ with $G \subseteq A$ such that $F \in \mathcal{H}(G)$. Infer from $\mathcal{H}(G) \in \tau_F$ and $\mathcal{H}(G) \subseteq \mathcal{H}(A)$ that $\mathcal{H}(G) \subseteq Int_F(\mathcal{H}(A))$. Since $F \in \mathcal{H}(G)$ we finally obtain that $F \in Int_F(\mathcal{H}(A))$.
\end{proof}

We use Lemma \ref{intF} to describe the boundary of the hitting-sets of a compact set in the Fell-topology $\tau_F$.\\

\begin{lemma} \label{boundary} If $K$ is compact in $(E,\mathcal{G})$, then:
\begin{itemize}
\item[(i)] $\partial_F \mathcal{H}(K)= \mathcal{H}(K) \setminus \mathcal{H}(K^0)$.
\item[(ii)] $\partial_F \mathcal{H}(K)= \{F \in \mathcal{F}: \emptyset \neq F \cap K \subseteq \partial K\}$.
\end{itemize}
\end{lemma}

\begin{proof} (i) By definition $\partial_F \mathcal{H}(K)=cl_F(\mathcal{H}(K)) \setminus Int_F(\mathcal{H}(K))$.
Here, $cl_F(\mathcal{H}(K))=\mathcal{H}(K)$, because $\mathcal{H}(K)$ is $\tau_F$-closed as $\mathcal{H}(K)^c = \mathcal{M}(K) \in \tau_F$.
Moreover, $Int_F(\mathcal{H}(K))=\mathcal{H}(K^0)$ by Lemma \ref{intF}, which now results in the equality (i).\\

(ii) By (i)
$$
 \partial_F \mathcal{H}(K)= \mathcal{H}(K) \cap \mathcal{M}(K^0) = \{F \in \mathcal{F}: F \cap K \neq \emptyset, F \cap K^0 =\emptyset\}.
$$
Consequently, we have to show: If $F \in \mathcal{F}$ satisfies $F \cap K \neq \emptyset$, then the following equivalence holds:
$$
 F \cap K^0 = \emptyset \quad \Leftrightarrow \quad F \cap K \subseteq \partial F
$$
To see the if-part let $x \in F \cap K$, i.e., $x \in F$ and $x \in K$. We have to prove that
$x \in \partial F = \overline{K} \setminus K^0 = K \cap (K^0)^c$. Since $x \in K$, it suffices to show that
$x \notin K^0$. Assume that $x \in K^0$. Then $x \in F \cap K^0$ in contradiction to $F \cap K^0 = \emptyset$.
For the only-if-part recall the equivalence $U \subseteq V \Leftrightarrow U \cap V^c = \emptyset$.
By assumption $F \cap K \subseteq \partial F$. Since $\partial F = K \cap (K^0)^c$ the equivalence yields
\begin{eqnarray*}
 \emptyset &=& F \cap K \cap (K \cap (K^0)^c)^c = F \cap K \cap (K^c \cup K^0)= (F \cap K \cap K^c)\cup (F \cap K \cap K^0)\\
 &=& F \cap K \cap K^0 = F \cap K^0
\end{eqnarray*}
\end{proof}

In the proof of Theorem \ref{thm1} the set $R(K)=\{r>0: P(\partial_F \mathcal{H}(K^r))=0\}$ plays an important role.\\

\begin{lemma} \label{dense} If $K$ is compact, then the complement $R(K)^c$ of $R(K)$ is at most countable.
As a consequence $R(K)$ lies dense in $[0,\infty)$.
\end{lemma}

\begin{proof} The key argument for the proof is to show that
the sets $\partial_F \mathcal{H}(K^r), r>0$, are pairwise disjoint.
Indeed, assume that this is not true. Then there exist two reals
$0<r<s$ such that $\partial_F \mathcal{H}(K^r) \cap \partial_F \mathcal{H}(K^s) \neq \emptyset$.
Therefore we find a set $F \in \mathcal{F}$ with $F \in \partial_F \mathcal{H}(K^r)$ and $F \in \partial_F \mathcal{H}(K^s)$.
By Lemma \ref{boundary} (ii) this means that $F$ satisfies the following two relations: (a) $\emptyset \neq F \cap K^r \subseteq \partial K^r$ and
(b) $\emptyset \neq F \cap K^s \subseteq \partial K^s$.

Now, $\partial K^r \subseteq \{x \in E: d(x,K)=r\}$ for all $r>0$. To see this first observe
that $\partial K^r =K^r \cap \overline{(K^r)^c}$ by Lemma \ref{Kr}. So, if $x \in \partial K^r$, then $d(x,K) \le r$ and there
exists a net $(x_\alpha)$ converging to $x$ with $d(x_\alpha,K)>r$ for all indices $\alpha$. But $d(\cdot,K)$ is continuous, whence
$d(x,K) \ge r$. Thus (a) and (b) imply (c) $\emptyset \neq F \cap K^r \subseteq \{x \in E: d(x,K)=r\}$ and (d) $\emptyset \neq F \cap K^s \subseteq \{x \in E: d(x,K)=s\}$.

By (c) there exists a point $x \in F \cap K^r$ with $d(x,K)=r$. Since $r<s$ and therefore $K^r \subseteq K^s$, $x$ a fortiori lies in $F \cap K^s$,
so that from (d) we can conclude that $d(x,K)=s$. It follows that $r=s$ in contradiction to $r<s$.

Next, observe that $R(K)^c =\{r>0: P(\partial_F \mathcal{H}(K^r))>0\}= \bigcup_{m \in \mathbb{N}} E_m$ with $E_m=\{r>0:P(\partial_F \mathcal{H}(K^r)) \ge 1/m\}$.
Here, $E_m$ contains at most $m$ elements, because otherwise we find at least $m+1$ positive numbers $r_1,\ldots,r_{m+1}$ with
$P(\partial_F \mathcal{H}(K^{r_j})) \ge 1/m$ for all $1 \le j \le m+1$. Herewith we arrive at
$$
 1 \ge P(\bigcup_{j=1}^{m+1} \partial_F \mathcal{H}(K^{r_j}))= \sum_{j=1}^{m+1} P(\partial_F \mathcal{H}(K^{r_j}) \ge (m+1) \frac{1}{m}>1,
$$
a contradiction. (Note that here the pairwise disjointness is essential, because it ensures the equality.)
Thus $R(K)^c = \bigcup_{m \in \mathbb{N}} E_m$ is denumerable. As to the second assertion of the lemma assume that
$R(K)^c$ is not dense in $[0, \infty)$. Then there exists a point $x \in [0,\infty)$ and a non-degenerate interval $I$ containing $x$ with
$I \cap R(K)=\emptyset$, which is the same as $I \subseteq R(K)^c$. It follows that $I$ is denumerable, a contradiction.
\end{proof}

If $\mathcal{A}$ is a family of subsets of $E$, then
\begin{equation} \label{Hsubset}
\mathcal{H}(\bigcap_{A \in \mathcal{A}} A) \subseteq \bigcap_{A \in \mathcal{A}}\mathcal{H}(A).
\end{equation}
Our next lemma gives a condition which ensures equality.\\

\begin{lemma} \label{Heq} If $(K_j)_{j \in \mathbb{N}}$ is a sequence of compact sets with $K_j \downarrow K \in \mathcal{K}$, then
$\mathcal{H}(K)= \bigcap_{j \in \mathbb{N}} \mathcal{H}(K_j)$.
\end{lemma}

\begin{proof} By (\ref{Hsubset}) it remains to show that $\bigcap_{j \in \mathbb{N}} \mathcal{H}(K_j) \subseteq \mathcal{H}(K)$.
So, let $F \in \bigcap_{j \in \mathbb{N}} \mathcal{H}(K_j)$. Then for every $j \in \mathbb{N}$ there exists a point $y_j \in F \cap K_j \ne \emptyset$. In particular, $(y_j)_{j \in \mathbb{N}}$ is a sequence in $K_1$, because (*) $K_1 \supseteq K_2 \supseteq \ldots$ by assumption.
Since $K_1$ is compact, $(y_j)_{j \in \mathbb{N}}$ has a convergent subsequence. For notational simplicity we assume that $y_j \rightarrow y \in K_1$. It follows from (*) that $y_j \in F \cap K_j \subseteq F \cap K_n$ for all $j \ge n$ and all $n \in \mathbb{N}$. Thus $(y_j)_{j \ge n}$ is a sequence in $F \cap K_n \in \mathcal{F}$ for all $n \in \mathbb{N}$. By closedness the limit $y$ lies in $F \cap K_n$ for all $n \in \mathbb{N}$, which in turn means that $y \in \bigcap_{n \in \mathbb{N}} (F \cap K_n) = F \cap \bigcap_{n \in \mathbb{N}} K_n = F \cap K$ by assumption. Consequently, $F \cap K$ is non-empty as it contains $y$ and therefore $F \in \mathcal{H}(K)$.
\end{proof}

Since $\mathcal{H}(\cdot)$ is monotone increasing with respect to $\subseteq$, the assertion in Lemma \ref{Heq} can be rewritten as
$\mathcal{H}(K_j) \downarrow \mathcal{H}(K)$.\\

Recall the family $\textbf{F}_{0,1}$ of all singletons inclusive the empty set. It is a Borel-set: $\textbf{F}_{0,1} \in \mathcal{B}_{F} = \sigma(\tau_F)$. This follows from the following lemma.\\

\begin{lemma} \label{F01} The set $\textbf{F}_{0,1}$ is closed in $(\mathcal{F},\tau_F)$.
\end{lemma}

\begin{proof} Since $(\mathcal{F},\tau_F)$ is metrizable, we can argue with sequences. So, let $(F_n)$ be a sequence in $\textbf{F}_{0,1}$ with
$F_n \rightarrow F$ in $(\mathcal{F},\tau_F)$. If $F = \emptyset$, then $F \in \textbf{F}_{0,1}$ and we are ready. Assume that $F \neq \emptyset$.
We have to prove that $F$ is a singleton. Let $G \in \mathcal{G}$ with $F \cap G \neq \emptyset$ (as for instance $G=E$). Then there exists a natural number $n_0$ such that $F_n \cap G \neq \emptyset$ for all $n \ge n_0$. Since all $F_n$ are either empty or a singleton, we now know that these
$F_n$ are singletons. Consequently, for every $n \ge n_0$ there exists a point $x_n \in E$ with $F_n=\{x_n\}$. Since convergence in the Fell-topology is equivalent with convergence in sense of Painlev\'{e}-Kuratowski, it follows that $F=\liminf_{n \rightarrow \infty} \{x_n\}$, where $\liminf_{n \rightarrow \infty} A_n$ denotes the lower limit of a sequence $(A_n)_{n \in \mathbb{N}}$ of sets, confer, e.g., Theorem C.7 in Molchanov \cite{Molchanov}.
It follows from the definition of the lower limit that $F$ is the set of all limit points of the sequence $(x_n)$. Now, $E$ is Hausdorff and therefore
$F$ is a singleton.
\end{proof}

Since $\tau_F \supseteq \tau_{uF}$, convergence in the Fell-topology entails that in the upper Fell-topology.
As to the reverse we have:\\

\begin{lemma} \label{reverse} Let $(F_\alpha)$ be a net in $\mathcal{F}$.
\begin{itemize}
\item[(a)] Assume that
  \begin{itemize}
  \item[(i)] $F_\alpha \rightarrow F=\{x\}$ in $(\mathcal{F},\tau_{uF})$ with $x \in E$.
  \item[(ii)] There exists a $K \in \mathcal{K}$ and $\alpha_0 \in A$ such that $\emptyset \neq F_\alpha \subseteq K$ for all $\alpha \ge \alpha_0$.
  \end{itemize}
Then $F_\alpha \rightarrow F$ in $(\mathcal{F},\tau_{F})$.
\item[(b)]
$$F_\alpha \rightarrow \emptyset \text{ in } (\mathcal{F},\tau_{uF}) \quad \Leftrightarrow \quad F_\alpha \rightarrow \emptyset \text{ in } (\mathcal{F},\tau_{F})$$
\end{itemize}
\end{lemma}

\begin{proof} (a) Let $\textbf{S} \in \mathcal{S}_F=\{\mathcal{M}(K):K \in \mathcal{K}\}\cup \{\mathcal{H}(G): G \in \mathcal{G}\}$ be a subbase-neighborhood of $F$. If $\textbf{S}=\mathcal{M}(K_0)$ with $K_0 \in \mathcal{K}$, then $\textbf{S} \in \tau_{uF}$, whence by (i) there exists an $\alpha_1 \in A$ such that
$F_\alpha \in \textbf{S}$ for all $\alpha \ge \alpha_1$. If $\textbf{S}=\mathcal{H}(G)$ with $G \in \mathcal{G}$, then $x \in G$. For $K_1:=K \setminus G$ with $K$ as in (ii) we know that it is compact and that $x \notin K_1$. Consequently $F = \{x\} \in \mathcal{M}(K_1) \in \tau_{uF}$.
By (i) there exists an $\alpha_2 \in A$ such that $F_\alpha \in \mathcal{M}(K_1)$ for all $\alpha \ge \alpha_2$. Conclude that
$$
 \emptyset = F_\alpha \cap K_1 = F_\alpha \cap (K \setminus G)=F_\alpha \cap K \cap G^c=F_\alpha \cap G^c \quad \forall \; \alpha \ge \alpha_0,\alpha_2,
$$
because $F_\alpha \cap K = F_\alpha$ by (ii). Herewith it follows that
$$
 \emptyset \neq F_\alpha= (F_\alpha \cap G) \cup (F_\alpha \cap G^c)=F_\alpha \cap G \quad \forall \; \alpha \ge \alpha_0,\alpha_2
$$
and thus $F_\alpha \in \mathcal{H}(G)=\textbf{S}$ for all $\alpha \ge \alpha_3$ with some $\alpha_3 \ge \alpha_0,\alpha_2$. Summing up we arrive at
$F_\alpha \rightarrow F$ in $(\mathcal{F},\tau_F)$.

(b) It remains to prove the implication $\Rightarrow$. But this follows immediately, because every subase-neighborhhod lies in $\mathcal{S}_{uF}$
as $\emptyset \notin \mathcal{H}(G)$ for every open $G$.
\end{proof}

\textbf{Notes}

If $E$ is a linear space with a metric $d$, then the statements in Lemmas 4-7 can be found in Salinetti and Wets \cite{SalWe}.
More precisely, Lemma \ref{intF} is presented in (1.9) on p. 389 in the special case that $A$ is compact, whereas we allow $A$ to be an arbitrary subset of $E$. Lemma \ref{boundary} is given in (1.10) on the same page, but without proof. Similarly, the statement of Lemma \ref{Kr} is a little hidden in the line directly before Corollary 1.13 on p.390, again without proof. Furthermore, our Lemma \ref{dense} coincides with the just mentioned Corollary 1.13.
Here, Salinetti and Wets \cite{SalWe} use a completely different technique to prove it, however the argument only works if $E$ is actually a normed linear space.
In addition, Lemma \ref{Heq} is used in Molchanov \cite{Molchanov} on p.7 without any justification. Finally, as far as Lemmata \ref{F01} and \ref{reverse} are concerned, due to their special character, we assume that these are new findings.

\vspace{1cm}
\textbf{Declarations}\\

\textbf{Compliance with Ethical Standards}: I have read and I understand the provided information.\\

\textbf{Competing Interests}: The author has no competing interests to declare that are relevant to the content of
this article.


\begin{thebibliography}{22}

\bibitem{Beer} G. Beer, \textit{Topologies on Closed and Closed Convex Sets}, Dordrecht, The Netherlands: Kluwer Academic Publishers, 1993.

\bibitem{Bill2} P. Billingsley, \textit{Convergence of Probability Measures}, Second Edition, New York: John Wiley \& Sons, 1999.

\bibitem{Engel} R. Engelking, \textit{General Topology, Revised and completed edition}, Berlin: Heldermann-Verlag, 1989.

\bibitem{Fell} J. Fell, \textit{A Hausdorff topology for the closed subsets of a locally compact non-Hausdorff space}, Proc. Amer. Math. Soc. \textbf{13} (1962), 472--476.

\bibitem{Ferger0} D. Ferger, \textit{Weak convergence in hyperspaces}, Lecture notes, Technische Universit\"{a}t Dresden, (2024).

\bibitem{Ferger1} D. Ferger, \textit{Weak convergence of probability measures to Choquet capacity functionals}, Turk. J. Math. \textbf{42} (2018), 1747--1764.

\bibitem{Stute} P. G\"{a}nssler and W. Stute, \textit{Wahrscheinlichkeitstheorie}, Berlin, Heidelberg: Springer-Verlag, 1977.

\bibitem{Gersch} O. Gersch, \textit{Convergence in Distribution of Random Closed Sets and Applications in Stability Theory of Stochastic Optimisation}, PhD thesis, Technische Universit\"{a}t Ilmenau, 2007.

\bibitem{Matheron} G. Matheron, \textit{Random Sets and Integral Geometry}, New York, London, Sydney, Toronto: John Wiley \& Sons.

\bibitem{Molchanov} I. Molchanov, \textit{Theory of random sets, Second Edition}, London: Springer-Verlag, 2017.

\bibitem{SalWe} G. Salinetti and R.J.-B. Wets, \textit{On the convergence in distribution of measurable multifunctions (random sets), normal integrands, stochastic processes and stochastic infima}, Math. Oper. Res. \textbf{11} (1986), 385--419.

\bibitem{Weil} R. Schneider and W. Weil, \textit{Stochastic and Integral Geometry}, Berlin, Heidelberg: Springer-Verlag, 2008.

\bibitem{Top} F. Tops{\o}e, \textit{Topology and Measure}, Lecture Notes in Mathematics Vol. 133, Berlin-Heidelberg-New York: Springer-Verlag, 1970.

\bibitem{Vaughan} H. E. Vaughan, \textit{On Locally  Compact Metrisable Spaces}, Bull. Amer. Math. Soc. \textbf{43} (1937), 532--535.

\bibitem{Vogel} S. Vogel, \textit{Semiconvergence in distribution of random closed sets with application to random optimization problems}, Ann. Oper. Res. \textbf{142} (2006), 269--282.
\end{thebibliography}

\end{document}